\documentclass[12pt]{article}
\usepackage{amsfonts}
\usepackage{amsmath}
\usepackage{latexsym}
\usepackage{amsthm}
\usepackage{amssymb}
\usepackage{graphicx}
\usepackage[dvips]{color}

\newtheorem{theorem}{Theorem}[section]

\newtheorem*{theorem*}{Theorem}

\newtheorem*{claim}{\hspace{20pt} Claim}

\newtheorem{corollary}[theorem]{Corollary}

\newtheorem{lemma}[theorem]{Lemma}

\newtheorem{proposition}[theorem]{Proposition}

\DeclareMathOperator{\esslim}{ess\lim\ }

\date{}
\author{{\bf Antonio J. Ure\~na}}
\title{Dynamics of periodic second-order equations between an ordered pair of lower and upper solutions\thanks{Supported by project MTM2008-02502, Ministerio de Educaci\'on y Ciencia, Spain, and FQM2216, Junta de Andaluc\'ia.}}
\begin{document}
\maketitle
\begin{abstract}
We consider periodic second-order equations having an ordered pair
of lower and upper solutions and show the existence of asymptotic
trajectories heading towards the maximal and minimal periodic solutions which lie
between them.
 \end{abstract}

{\em Key words:}{ Lower and upper solutions, dynamics, asymptotic solutions, instability.}

 \section{Introduction}
 Let $f=f(t,u,\dot u)$ be $T$-periodic in time and consider the
 second-order equation
 \begin{equation}\label{eu0}
-\ddot u=f(t,u,\dot u)\,.
\end{equation}

  A well-known strategy to find $T-$periodic solutions is the so-called lower and upper solutions method. Roughly speaking, this approach requires finding
   periodic functions $\alpha\leq\beta$ with $-\ddot\alpha\leq f(t,\alpha,\dot\alpha)$ and $-\ddot\beta\geq f(t,\beta,\dot\beta)$, and under some
   conditions on the dependence of $f$ with respect to $\dot u$, it guarantees the
  existence of some periodic solution $x$  between them.

As a model, we may think on the  $T$-periodic problem for the {\em
damped pendulum equation}:

\begin{equation}\label{eu00}
-\ddot u=c\,\dot u+a\sin u\,,
\end{equation}
where $c\geq 0$ and $a>0$ are given parameters. We observe that
$$\alpha(t)\equiv\frac{\pi}{2}\,,\hspace{2cm}\beta(t)\equiv\frac{3\pi}{2}$$ are, respectively,
lower and upper solutions. They enclose the $T-$periodic solution
$x(t)\equiv\pi$, which is unstable. Moreover, this equation possesses other solutions which are asymptotic to $x$; infinitely many in the
past and also infinitely many in the future.

A second look at this example shows that
$$\hat\alpha(t)\equiv\frac{\pi}{2}-2\pi\,,\hspace{2cm}\hat\beta(t)\equiv\frac{3\pi}{2}+2\pi$$ also make up an
ordered pair of lower and upper solutions. This time there are
several periodic solutions, with different dynamics, between them.
For instance, $x(t)\equiv 0$ is stable in the future (throughout
this paper stability is understood in the the Lyapunov sense) . But
the {\em minimal} and the {\em maximal} periodic solutions of
(\ref{eu00}) which lie between $\hat\alpha$ and $\hat\beta$ are,
respectively,
$$x_{\min}(t)\equiv-\pi\,,\hspace{2cm} x_{\max}(t)\equiv3\pi\,,$$ and both of them
are unstable (we call instability to the logical negation of Lyapunov stability). Indeed, both $x_{\min}$ and $x_{\max}$ are again the limit of infinitely many
asymptotic solutions.

In this paper we generalize this fact for more general equations
(\ref{eu0}) having an ordered pair $\alpha\leq\beta$ of periodic
lower and upper solutions. Assuming that, for instance, $\alpha$
does {\bf not} solve our equation (\ref{eu0}) we show that the minimal periodic solution $x_{\min}$ lying
between $\alpha$ and $\beta$ is unstable; moreover, it is the limit
of many asymptotic solutions, in the past and in the
future  (Theorem \ref{th1} and Corollary \ref{cor1}). Of course, a similar statement holds for the maximal
periodic solution $x_{\max}$ if $\beta$ does not solve (\ref{eu0}).

This result raises the question about the case in which
$\alpha<\beta$ are ordered $T-${\em periodic solutions} of our
equation (\ref{eu0}). Is it still possible to find an unstable solution between them? The answer is affirmative in the conservative case $-\ddot u=f(t,u)$ (Corollary \ref{cor3}). For general, derivative-depending equations it may not true, but we shall see
that when the periodic solutions are {\em neighboring}, then one of
them must be unstable and  possess many asymptotic
solutions (Theorem \ref{th2} and Corollary \ref{cor2}).

The instability of the periodic solutions obtained by the method of
lower and upper solutions was previously studied by Dancer and
Ortega \cite{dan-ort}. When the lower and upper solutions are
strict and the number of periodic solutions between them is finite,
they showed that at least one of them must be unstable. This result
was obtained as an application of a general theorem on the index of
stable fixed points in two dimensions whose proof depended on a
nonelementary result of planar topology: Brouwer's lemma on
translation arcs.

 Even though the main result of this paper can be seen as a generalization of Proposition 3.1 of \cite{dan-ort}, our approach, which
 is inspired on the Aubry-Mather theory as presented by Moser \cite{mos}, is based on the more elementary concepts of maximal and minimal solutions of Dirichlet and periodic boundary value problems. Moreover, we do not need to assume the finiteness of the set of periodic solutions comprised between the lower and upper solutions and we shall obtain more precise information on the dynamics: the existence of branches of asymptotic solutions.

This paper is distributed in five sections. In Section \ref{sec2} we
make precise the framework which we shall use and collect some basic
facts on upper and lower solutions which will be needed in the proofs. We do not
pretend any originality in this section, whose contents are mostly
well-known in slightly different frameworks. To keep the pace of the exposition, the proofs of these results will be postponed to the Appendix (Section \ref{sec5}). The
main results of this paper are stated and proved in Sections
\ref{sec3} (where we study the dynamics between an ordered pair of lower and upper solutions) and \ref{sec4} (devoted to the dynamics between two ordered solutions).

Some remarks on the notation (previously employed). Through most of the paper the Greek letters
$\alpha$ and $\beta$ will be kept, respectively, for lower and upper
solutions; however, in Section \ref{sec4} they will represent two ordered periodic solutions. We shall also utilize different notations to distinguish
solutions depending on the kind of  boundary conditions that they satisfy.
In this way, periodic solutions will be usually written as $x$,
the letter $y$ will be used to emphasize that our solution satisfies some Dirichlet conditions, and $u$
will stand for other solutions, not necessarily verifying
any particular boundary condition. Finally, it will be convenient to
adopt the following convention: given functions $h_1,h_2$ defined on
respective domains $I_1,I_2\subset\mathbb R$, and some set $J\subset
I_1\cap I_2$, we shall simply say $h_1\leq h_2$ (resp., $h_1<h_2$) on $J$ instead of
$h_1(t)\leq h_2(t)$ (resp., $h_1(t)< h_2(t)$) for any $t\in J$. When there is no possible ambiguity on the domain we shall say $h_1\leq h_2$ or $h_1<h_2$ meaning that the inequality holds pointwise. Given functions $u,v,w:I\to\mathbb R$ we shall say that {\em $w$ lies between $u$ and $v$ on $I$} to mean $u\leq w\leq v$ on $I$.

 It was shown in \cite{ure} that periodic minimizers are unstable. I am indebted to Prof. P. Omari for pointing out to me  that it implies the instability of periodic solutions given by the lower and upper solutions method when the problem has a variational structure and formulating the question about the nonvariational case.

Last but not least, I want to express my gratitude to Prof. R.
Ortega directing me to Moser's book \cite{mos}, as well as for his invaluable advice on many aspects of this paper,
including the regularity result on lower and upper solutions
described in the Appendix.

 \section{Some basic facts on lower and upper solutions}\label{sec2}

In this paper, the nonlinearity $f:\mathbb R^3\to\mathbb R$,
$(t,u,\dot u)\mapsto f(t,u,\dot u)$ will {\em always} be assumed to
be continuous and $T$-periodic in time. The function $\alpha:\mathbb
R\to\mathbb R$ is called a {\em lower solution} of the periodic
problem
\begin{equation*}
(P)\equiv\begin{cases} -\ddot x=f(t,x,\dot x)\,,\vspace{0.3cm}\\
x(t)=x(t+T)\ \text{ for any } t\in\mathbb R\,.
\end{cases}
\end{equation*}
provided that it is Lipschitz-continuous and $T$-periodic, and
verifies the inequality
 \begin{equation}\label{eu1}
    -\ddot\alpha\leq f(t,\alpha,\dot\alpha)\,,
    \end{equation}
in the distributional sense on the real line. With other words,
\begin{equation}\label{eu2} \int_{-\infty}^\infty\dot\alpha(t)\dot\varphi(t)\,dt\leq\int_{-\infty}^\infty f(t,\alpha(t),\dot\alpha(t))\varphi(t)\,dt\,,
\end{equation}
for any nonnegative $\varphi\in \mathcal C^\infty(\mathbb R)$ having compact support. Upper solutions $\beta$  are defined similarly after changing the sign of inequalities (\ref{eu1},\ref{eu2}).

Of course, if $\alpha$ is a  $\mathcal C^2$ function, then (\ref{eu2}) is
equivalent to (\ref{eu1}) holding pointwise. But the
definition above includes also some functions displaying angles. We
point out an example which will be used later. Let
$\alpha:[0,T]\to\mathbb R$ be a solution of our equation (\ref{eu0})
such that $\alpha(0)=\alpha(T)$ but $\dot\alpha(0)>\dot\alpha(T)$,
and extend it to the real line by periodicity. The resulting,
piecewise $\mathcal C^2$ function, is a lower solution of $(P)$, as one
may easily check using an integration-by-parts argument.

Notice also
that our concepts of lower and upper solutions are nonstrict;
periodic solutions are examples of lower and upper solutions.

The distributional notion of lower and upper solutions adopted in
this paper is routinely used in PDE problems but may seem unnecessarily general here. Indeed, our main
arguments could be carried out using only piecewise-$\mathcal C^2$
lower and upper solutions, even though the results obtained in this way would lose some generality. Our choice will imply some relatively small variations with respect to other versions in the literature when showing the results of this Section, but otherwise it will not lead to additional difficulties in the proofs. Incidentally we observe that lower and upper solutions in the distributional sense must be somewhat more regular than merely Lipschitz-continuous; for instance, the derivative should have side limits at each point (it will be shown in the proof of Lemma \ref{ML} in the Appendix).

If  $f=f(t,u)$ does not depend on $\dot u$, then the existence of an
ordered pair $\alpha\leq\beta$ of lower and upper solutions for
$(P)$ implies the existence of a solution between them. However,
this result fails to hold for general, derivative-dependent
nonlinearities as considered in this paper, see \cite{dec-hab}, Example 4.1, pp. 43-44. To ensure that the
method of lower and upper solutions works for $(P)$ one must add
some further condition, which may consist on a special form of our
equation (the Rayleigh equation, the Li\'enard equation...) or some
assumption on the growth of $f$ with respect to $\dot u$ (such as
the Bernstein condition or some kind of Nagumo condition); see \cite{dec-hab}, Chapter I, Section 4 for a detailed study of all these possibilities. In this paper we opt for the classical, two-sided Nagumo condition:
\begin{enumerate}
\item[{\bf [N]}]  There exists a continuous function $\varphi:[0,+\infty)\to(0,+\infty)$ with $$\int_0^{+\infty}\frac{v}{\varphi(v)}\,dv>\max_{\mathbb R}\beta-\min_{\mathbb R}\alpha\,,$$ and such that  $|f(t,u,\dot u)|\leq\varphi(|\dot u|)$ if $\alpha(t)\leq u\leq\beta(t)\,.$
\end{enumerate}

This assumption makes the usual existence result for the periodic method of lower and upper solutions hold:

\begin{proposition}[Upper and lower solutions method for periodic problems]\label{prop0}
{Let $(P)$ have an ordered pair $\alpha\leq\beta$ of lower and upper
solutions and assume {\bf [N]}. Then there exists a solution $x$ of
$(P)$ with $\alpha\leq x\leq\beta$.}
\end{proposition}

Lower and upper solutions are nowadays standard notions, and have
been studied in depth. We shall be particularly interested on the
structure of the set of solutions of $(P)$ lying between a given
ordered pair $\alpha\leq\beta$ of lower and upper solutions. This
set always contains a minimal and a maximal solution, and the
following improvement of Proposition \ref{prop0} is well-known:

\begin{proposition}\label{upp-low}{Let $(P)$ have an ordered pair $\alpha\leq\beta$ of lower and upper solutions and assume {\bf [N]}. Then there are solutions $x_{\min},\ x_{\max}$ of $(P)$ with
$$\alpha\leq x_{\min}\leq x_{\max}\leq\beta\,,$$
and such that any third solution $x$ of $(P)$ with $\alpha\leq
x\leq\beta$  verifies
$$x_{\min}\leq x\leq x_{\max}\,.$$ }
\end{proposition}

Since the dependence of $f$ with respect to $(u,\dot u)$ was only
required to be continuous, initial value problems associated to our
equation (\ref{eu0}) may have many periodic solutions. Sometimes we shall merely look for global branches of asymptotic solutions to a given periodic motion and it will be not problematic. But other results of this paper explicitly deal with the instability of certain periodic solutions; in these cases, uniqueness for initial value problems will be assumed in combination with {\bf [N]}. We emphasize a connection between these assumptions
and the  method of lower and upper solutions:

\begin{proposition}\label{prop1}{Assume that there is uniqueness for initial value problems associated to (\ref{eu0})
and $\alpha\leq\beta$ are lower and upper solutions of $(P)$.
Finally, assume also {\bf [N]}. Then either $\alpha(t)<\beta(t)$ for
any $t\in\mathbb R$ or $\alpha\equiv\beta$.}
\end{proposition}
Concerning this result we remark that the lower and upper solutions $\alpha,\beta$ are {\bf not} assumed to be strict. In particular, it also applies when either $\alpha$ or $\beta$ are periodic solutions. We shall use this form of Proposition \ref{prop1} in the proof of Corollary \ref{cor1}.

In this paper we shall also consider Dirichlet-type  boundary value problems:
\begin{equation*}
 (D)\equiv\begin{cases}
-\ddot y=f(t,y,\dot y)\,,\vspace{0.3cm}\\  y(a)=y_a\,,\ \
y(b)=y_b\,,
\end{cases}
\end{equation*}
where $a<b$ and $y_a,y_b\in\mathbb R$ are given. By a lower solution  to $(D)$ we mean a Lipschitz-continuous function $\alpha:[a,b]\to\mathbb R$ such that
\begin{equation}\label{eu4}
\alpha(a)\leq y_a,\hspace{2cm} \alpha(b)\leq y_b\,,
 \end{equation}
 and (\ref{eu1}) holds in the distributional sense on $(a,b)$. With other words, a Lipschitz-continuous function $\alpha$ verifying (\ref{eu4}) and such that (\ref{eu2}) holds for any nonnegative $\varphi\in \mathcal C^\infty(\mathbb R)$ with compact support contained in $(a,b)$. Upper solutions $\beta$ for this problem are then defined similarly by changing the sign of the inequalities.

In this framework it is also well-known that the presence of an
ordered pair of lower and upper solutions does not necessarily imply
the existence of a solution between them. We might use a related Nagumo
condition, but for the purposes of this
paper it will suffice to consider the more restrictive class of
nonlinearities which are bounded between $\alpha$ and $\beta$, i.e., there exists some constant $M>0$ such that
\begin{equation}\label{eu10}
|f(t,u,\dot u)|\leq M\ \ \text{ if }\alpha(t)\leq u\leq\beta(t)\end{equation}
for any  $t\in[a,b]$.
\begin{proposition}[Lower and upper solutions method for Dirichlet
problems]\label{DP}{Let (D) have an ordered pair $\alpha\leq\beta$
of lower and upper solutions and let $f$ be bounded between them. Then there are solutions $y_{\min},\
y_{\max}$ of (D) with $$\alpha\leq y_{\min}\leq y_{\max}\leq\beta\
\text{ on }[a,b]\,,$$ and such that any third solution $y$ of (D)
with $\alpha\leq y\leq\beta$ on $[a,b]$ verifies
$$y_{\min}\leq y\leq y_{\max}\ \text{ on }[a,b]\,.$$}
               \end{proposition}

    For solutions of Dirichlet problems, the quality of being minimal or maximal is inherited after restriction to a smaller interval.
To state this result in a precise form, let us assume that
$\alpha\leq\beta$ are ordered lower and upper solutions of $(D)$ and
$y:[a,b]\to\mathbb R$ is a solution between them. Given $a\leq\tilde
a<\tilde b\leq b$, the restriction of $y$ to $[\tilde a,\tilde b]$
solves the Dirichlet problem
      \begin{equation*}(\tilde D)\equiv
 \begin{cases}
-\ddot y=f(t,y,\dot y)\,,\vspace{0.3cm}\\  y(\tilde a)=\tilde
y_a\,,\ \ y(\tilde b)=\tilde y_b\,,
\end{cases}
\end{equation*}
for suitable choices of $\tilde y_a$ and $\tilde y_b$. Also, the
restrictions of $\alpha$ and $\beta$ to $[\tilde a,\tilde b]$ become
lower and upper solutions for $(\tilde D)$.

\begin{proposition}\label{prop}{Let $f$ be bounded between $\alpha$ and $\beta$ and let $y$ be the minimal/maximal solution of $(D)$ lying between them. Then,
$y_{\big|[\tilde a,\tilde b]}$ is the minimal/maximal solution of
($\tilde D$) lying between ${\alpha}_{\big|[\tilde a,\tilde b]}$ and
 ${\beta}_{\big|[\tilde a,\tilde b]}$ }
\end{proposition}

We shall need also some facts on the relative compactness on the $\mathcal C^2$ topology of some sets of solutions of our equation. The next result is not directly related with the method of lower and
upper solutions but follows from the boundedness assumption on $f$. To present it precisely, we choose numbers $a<b$ and assume
that the sequence  $u_n:[a,b]\to\mathbb R$ of solutions of
(\ref{eu0}) is uniformly bounded and pointwise converging. With other words,
\begin{equation}\label{eu9}
\{u_n(t)\}_n\ \text{ converges for every } t\in[a,b]\,,\hspace{1.5cm}(t,u_n(t))\in\mathcal B\ \text{ for all }t\in[a,b]\text{ and }n\in\mathbb N\,,\end{equation}
where $\mathcal B$ is some compact region of $[a,b]\times\mathbb R$.
  \begin{proposition}\label{prp1} {Assume that $f$ is bounded on $\mathcal B\times\mathbb R$ and $(\ref{eu9})$. Then, $u(t):=\lim_{n\to+\infty}u_n(t)$ is again a solution of $(\ref{eu0})$ and $\{u_n\}\to u$ in the
  $\mathcal C^2[a,b]$ topology. Precisely,
  $$\lim_{n\to\infty}(u_n(t)-u(t))=0,\hspace{1.2cm}\lim_{n\to\infty}(\dot u_n(t)-\dot u(t))=0,\hspace{1.2cm}\lim_{n\to\infty}(\ddot u_n(t)-\ddot u(t))=0\,,$$
  uniformly with respect to $t\in[a,b]$.
  }

\end{proposition}
Some results of this paper will be first shown when the equation is bounded between a given pair $\alpha\leq\beta$ of periodic lower and upper solutions, i.e., assuming that (\ref{eu10}) holds for some constant $M>0$. This assumption will be subsequently relaxed to
  {\bf [N]} by replacing (\ref{eu0}) with a modified equation
  \begin{equation*} -\ddot u=\tilde f(t,u,\dot u)\tag{$\tilde 1$}
  \end{equation*}
  which will be bounded between $\alpha$ and $\beta$ and share many periodic solutions with the original equation. For this reason, the next Lemma  will be evoked several times  throughout this paper.
\begin{lemma}[Modification Lemma]\label{ML}{Let (P) have an ordered pair $\alpha\leq\beta$ of lower and upper solutions and assume {\bf [N]}. Then,  there exists some continuous and $T$-periodic in time $\tilde f:\mathbb R^3\to\mathbb R$  which is bounded between $\alpha$ and $\beta$ and verifies: \begin{enumerate}
     \item[(i)] $\alpha$ and $\beta$ remain, respectively, lower and upper solutions for the $T$-periodic problem associated to $(\tilde 1)$.
   \item[(ii)] Any solution $u:I\to\mathbb R$ of $(\tilde 1)$ for which $u-\alpha$ attains a local minimum
   at some interior point $t_0$ of $I$ satisfies $u(t_0)\geq\alpha(t_0)$. Similarly, given a solution $u:I\to\mathbb R$ of $(\tilde 1)$ such that $u-\beta$ attains a local maximum
   at some interior point $t_0$ of $I$ one has  $u(t_0)\leq\beta(t_0)$.
     \item[(iii)]
   There exists some $\epsilon>0$, not depending on $t_0, u,$ nor the interval $I$, such that $(\ref{eu0})$ and $(\tilde 1)$ have exactly the same solutions $u:I\to\mathbb R$ with $$\alpha\leq u\leq\beta\text{ on }I\,,\hspace{2cm}|\dot u(t_0)|<\epsilon\text{ for some }t_0\in I\,.$$
 \end{enumerate}}
  \end{lemma}

We notice the following consequence of {\em (ii)-(iii)}: the
periodic solutions of the modified equation $(\tilde 1)$ are the
periodic solutions of the original equation $(1)$ lying between
$\alpha$ and $\beta$. In particular, the maximal and the minimal
$T$-periodic solutions between $\alpha$ and $\beta$ coincide for
both equations.

 The results described in this Section will be discussed with some detail in the Appendix. Next, we present and prove the main results of this paper.

  \section{Maximal or minimal periodic solutions and asymptotic trajectories}\label{sec3}
Our starting point will be the periodic problem $(P)$, which through this
Section we assume to have {\em  an ordered pair
$\alpha\leq\beta$ of lower and upper solutions}. As usually, the solution $u_r:[0,+\infty)\to\mathbb R$ of
(\ref{eu0}) will be called {\em asymptotic in the future} to the periodic
solution $x$ if
\begin{equation}\label{eu17}
\lim_{t\to+\infty}\big(u_r(t)-x(t)\big)=\lim_{t\to+\infty}\big(\dot u_r(t)-\dot
x(t)\big)=0\,.
\end{equation}

Similarly, one may speak of solutions $u_l:(-\infty,0]\to\mathbb R$
which are asymptotic in the past to a given periodic solution.

  \begin{theorem}\label{th1}{Assume {\bf [N]}. Assume also that, for instance,
$\alpha$ is {\bf\em not} a solution of $(P)$ and let $x_{\min}$ be the
minimal solution of $(P)$ lying between $\alpha$ and $\beta$. Then,
for any initial position $$u_0\in\Big[\alpha(0),x_{min}(0)\Big)\,,$$
there are solutions $u_l:(-\infty,0]\to\mathbb R,\ \
u_r:[0,+\infty)\to\mathbb R$ of (\ref{eu0}) with
\begin{equation}\label{eu8}
u_l(0)=u_0=u_r(0)\,,\hspace{1.3cm}\alpha\leq u_l\leq x_{\min}\ \text{ on }(-\infty,0]\,,\hspace{1.3cm}\alpha\leq u_r\leq x_{\min}\ \text{ on }[0,+\infty)\,,
\end{equation}
which are asymptotic to $x_{\min}$ in the past and in the future
respectively. See Fig. 1 below.}

\begin{figure*}[!h!]\label{fig1}
\begin{center}
\includegraphics[scale=0.6]{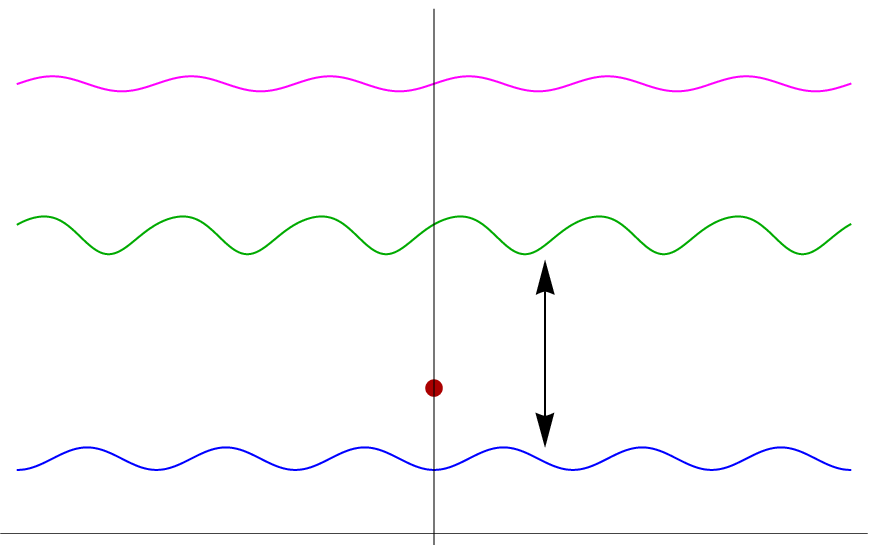}\put(-150, 65.2){\small $x_{\min}$}\put(-150, 88.2){\small $\beta$}\put(-150, 22){\small $\alpha$}\put(-6,-2){\footnotesize $t\rightarrow$}\put(-82, 93){\footnotesize $u$}\put(-82, 103){\footnotesize $\uparrow$}\put(-52, 41){\footnotesize no periodic sols.}\put(-75, -2){\footnotesize $0$}\put(-44, 29){\footnotesize in between}\put(-88, 30.6){\footnotesize $u_0$}\includegraphics[scale=0.6]{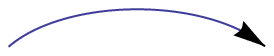}\includegraphics[scale=0.6]{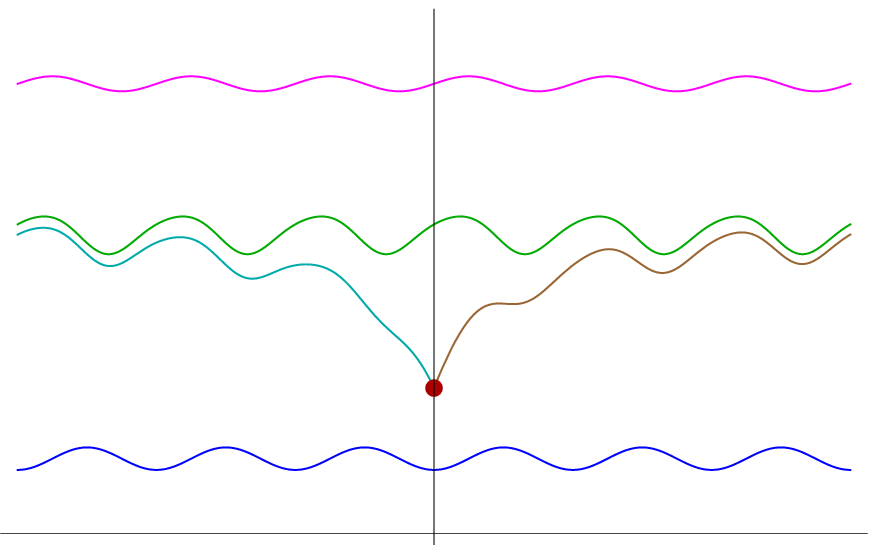} \put(-150, 65.2){\small $x_{\min}$}\put(-150, 88.2){\small $\beta$}\put(-150, 22){\small $\alpha$}\put(-42.8, 43.2){\small $u_r$}\put(-121.4, 45.6){\small $u_l$}\put(-73.8, -2){\footnotesize $0$}\put(-6,-2){\footnotesize $t\rightarrow$}\put(-82, 93){\footnotesize $u$}\put(-82, 103){\footnotesize $\uparrow$} \put(-88, 30.6){\footnotesize $u_0$}
\end{center}\caption{If $\alpha$ is not a solution, then $x_{\min}$ is the limit of many asymptotic solutions.}
\end{figure*}
\end{theorem}

\noindent {\bf Remark:} We recall that the (global) {\em stable manifold}  of the periodic solution $x$ is the set $W^S(x)$ of initial conditions of solutions $u_r$ which are asymptotic to $x$ in the future:
$$W^S(x):=\Big\{\big(u_r(0),\dot u_r(0)\big)\text{ such that the solution }u_r:[0,+\infty)\to\mathbb R\text{ of }(\ref{eu0})\text{ satisfies }(\ref{eu17})\Big\}\,.$$
Similarly, the {\em unstable manifold}  associated to $x$ is the set $W^U(x)$ of initial conditions of solutions $u_l$ which are asymptotic to $x$ in the past. These sets allow us to reformulate part of the information given by Theorem \ref{th1} as follows:

\vspace{0.3cm}

\noindent{\bf Theorem \ref{th1}$.^*$}{\em Assume {\bf [N]}. Assume also that $\alpha$ is not a solution of $(P)$, and let $x_{\min}$ be the minimal solution of $(P)$ lying between $\alpha$ and $\beta$. Then, for any initial position $u_0\in[\alpha(0),x_{\min}(0))$ there are initial velocities $v_l,v_r\in\mathbb R$ such that
$$(u_0,v_l)\in W^U(x_{\min}),\hspace{2cm}(u_0,v_r)\in W^S(x_{\min})\,.$$ }
 If $f$ satisfies some smoothness assumptions and the periodic solution $x$ is hyperbolic, then the Stable manifold Theorem states that $W^U(x)$ and $W^S(x)$ are injectively immersed smooth curves, see e.g. \cite{pal-mel}, Theorem 6.2. However, this result does not always apply to $x_{\min}$ under the assumptions of Theorem \ref{th1}, because $x_{\min}$ may not be hyperbolic. Indeed, one may construct a periodic in time, $\mathcal C^{0,2}$ equation $-\ddot x=f(t,x)$ which is {\em repulsive} in the sense that
\begin{equation}\label{eu18}
f(t,0)=0\,,\hspace{2cm}f(t,-x) > 0 > f(t, x)\ \forall x>0\,,
\end{equation}
and such that the stable manifold $W^S(x_*)$ associated to the periodic solution $x_*\equiv 0$ is pathological\footnote{Precisely, it may be neither arcwise connected nor locally connected, and may contain points which are not accessible.}, see \cite{ure2}. Observe that (\ref{eu18}) implies that every negative constant is a lower solution of the associated periodic problem, every positive constant is an upper solution, and $x_*\equiv 0$ is the only periodic solution of the equation.

\begin{proof}[Proof of Theorem \ref{th1}] After replacing $t$ with $-t$ one observes that it suffices to prove the part of the statement concerning the solutions $u_r$ which are asymptotic to $x_{\min}$ in the future. The proof will be divided into two steps. Notice that Step 1 and Phase 2(a) assume the boundedness of $f$ between the lower and the upper solutions.

\vspace{0.2cm}

     {\bf (Step 1)} {\em Let $f$ be bounded between $\alpha$ and $\beta$. Then, there exists some solution $u:[0,+\infty)\to\mathbb R$ which is asymptotic to $x_{\min}$ in the future and
     at initial time $t=0$ starts from the position $u_0=\alpha(0)$.}

\vspace{0.2cm}

To show this Step we consider, for each $n\in\mathbb N$, the Dirichlet
problem

\begin{equation*}
 (D_n)\equiv\begin{cases}
-\ddot y=f(t,y,\dot y)\,,\vspace{0.3cm}\\  y(0)=\alpha(0)\,,\ \
y(nT)=x_{\min}(nT)\,.
\end{cases} \end{equation*}

Observe that $\alpha$ is a lower solution for this problem, while
$x_{\min}$ is an upper solution. Furthermore $\alpha\leq
x_{\min}$. Then, Proposition \ref{DP} states the existence of some
solution of $(D_n)$ which lies between $\alpha$ and $x_{\min}$. This
solution may not be unique, but we choose the maximal one and call
it $y_n$. In this way, $y_n$ is a $\mathcal C^2$ function on $[0,nT]$. We
organize the proof of Step 1 in four phases:

\vspace{0.3cm}

{\em (Phase 1a)}\hspace{2cm} $y_n(t-T)\leq y_n(t)\,,\hspace{2cm} t\in[T,nT]\,.$

\vspace{0.3cm}

\noindent To check this assertion we use a contradiction argument and assume
instead that the contrary happens for some $t_*\in[T,nT]$. We
denote $\tilde y_n(t):=y_n(t-T)$, which solves our equation on
$[T,nT]$ and verifies:
$$\tilde y_n(T)\leq y_n(T)\,,\hspace{1cm}\tilde
y_n(nT)\leq y_n(nT)\,,\hspace{1cm}\tilde y_n(t_*)>y_n(t_*)\,,$$
and this implies the existence of numbers $T\leq\tilde
a<t_*<\tilde b\leq nT$ such that
$${\tilde y}_n(\tilde a)=y_n(\tilde a)\,,\hspace{2cm}{\tilde y}_n(\tilde b)=y_n(\tilde b)\,,$$
(see Fig. 2(a)). In view of  Proposition \ref{prop}, ${y_n}_{\big|[\tilde a,\tilde b]}$
should be the maximal solution of the Dirichlet problem
\begin{equation*}(\tilde D)\equiv
 \begin{cases}
-\ddot y=f(t,y,\dot y)\,,\vspace{0.3cm}\\  y(\tilde a)=y_n(\tilde a)\,,\ \
y(\tilde b)=y_n(\tilde b)\,.
\end{cases}
\end{equation*}
 But ${\tilde y_n}\hspace{1cm}\hspace{-1cm}_{\big|[\tilde a,\tilde
b]}$ also solves $(\tilde D)$  and at time $t_*$ it is greater
than $y_n$. This contradiction shows Phase {\em 1a}.

\begin{figure*}[!h!]\label{figg2}\begin{center}
\includegraphics[scale=0.7]{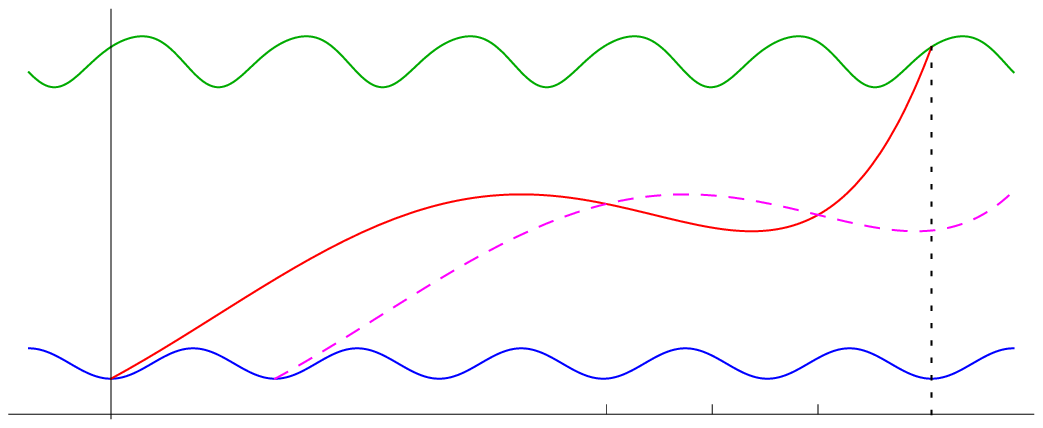}
\put(-18,87){\small $x_{\min}$}\put(-11,21){\small $\alpha$}\put(-133,59){\small $y_n$}\put(-11,53){\small $\tilde y_n$}\put(-35,-1){\scriptsize $nT$}\put(-55,-1){\scriptsize $\tilde b$}\put(-79,-1){\scriptsize $t_*$}\put(-100,-1){\scriptsize $\tilde a$}\put(-200,-1){\scriptsize $0$}\put(-167,-1){\scriptsize $T$}\put(-13,-1){\footnotesize $t\rightarrow$}\put(-205,88){\footnotesize $u$}\put(-205,98){\footnotesize $\uparrow$}\put(-100,-25){\small (a)}\hspace{1cm}  \includegraphics[scale=0.7]{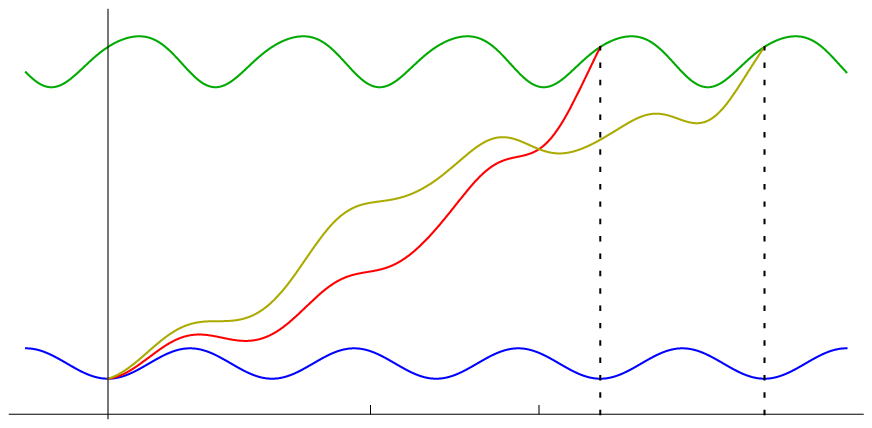}\put(-5,-1){\footnotesize $t\rightarrow$}\put(-18,87){\small $x_{\min}$}\put(-9,21){\small $\alpha$}\put(-44,-1){\scriptsize $(n$}\put(-36,-1){\scriptsize $+$}\put(-29,-1){\scriptsize $1)T$}\put(-65,-1){\scriptsize $nT$}\put(-164,-1){\scriptsize $0$}\put(-77,-1){\scriptsize $\tilde b$}\put(-110,-1){\scriptsize $t_*$} \put(-169,88){\footnotesize $u$}\put(-169,98){\footnotesize $\uparrow$}\put(-98,35){\small $y_n$}\put(-119,59){\small $y_{n+1}$}\put(-100,-25){\small (b)} \end{center}\caption{(a) The solutions $y_n$ must increase on time intervals of length $T$. (b) If $y_{n+1}$ is greater than $y_n$ at some point $t_*$, then  $y_n$ is not maximal on the interval $[0,\tilde b]$.}
\end{figure*}

\vspace{0.3cm}

 {\em (Phase 1b)}\hspace{2cm}$y_n\geq y_{n+1}\ \text{ on }[0,nT]\,.$

\vspace{0.3cm}

The proof of this new assertion uses similar ideas to the previous one. Since,
should it fail to be true, we could find some time $t_*\in[0,nT]$ with
$y_n(t_*)<y_{n+1}(t_*)$. But $y_n(nT)=x_{\min}(nT)\geq y_{n+1}(nT)$,
and we deduce the existence of some $\tilde b\in(t_*,nT]$ such that
$y_n(\tilde b)=y_{n+1}(\tilde b)$, see Fig. 2(b).

Since furthermore $y_n(0)=\alpha(0)=y_{n+1}(0)$, Proposition \ref{prop} states that both ${y_n}_{\big|[0,\tilde b]}$ and
${y_{n+1}}_{\big|[0,\tilde b]}$ should be the maximal
solution of the Dirichlet problem
\begin{equation*}(\tilde D)\equiv
 \begin{cases}
-\ddot y=f(t,y,\dot y)\,,\vspace{0.3cm}\\  y(0)=\alpha(0)\,,\ \
y(\tilde b)=y_n(\tilde b)\,.
\end{cases}
\end{equation*}
 Hence, these
solutions must be constantly equal, contradicting the existence of
$t_*$.

\vspace{0.3cm}

{\em (Phase 1c)}\hspace{2cm}There exists some solution $u:[0,+\infty)\to\mathbb
R$ of (\ref{eu0}) with
\begin{equation}\label{eu7}
u(0)=\alpha(0)\,,\hspace{0.5cm}\alpha\leq u\leq x_{\min}\ \text{ on
}[0,+\infty)\,,\hspace{0.5cm}u(t)\leq u(t+T)\ \forall
t\in[0,+\infty)\,.
\end{equation}
To see this, we apply Proposition \ref{prp1} to the decreasing
sequence $u_n:=y_n$, and obtain that it converges in the $\mathcal
C^2$ topology on compact sets to some solution
$u:[0,+\infty)\to\mathbb R$ of (\ref{eu0}), see Fig. 3(a). Now, since all functions
$y_n$ lie between $\alpha$ and $x_{\min}$ and start from $\alpha(0)$
at time $t=0$, the same happens with $u$. The last assertion of
(\ref{eu7}) is a direct consequence of Phase {\em 1a}.

\begin{figure*}[!h!]\label{fig3}
\begin{center}
\includegraphics[scale=0.7]{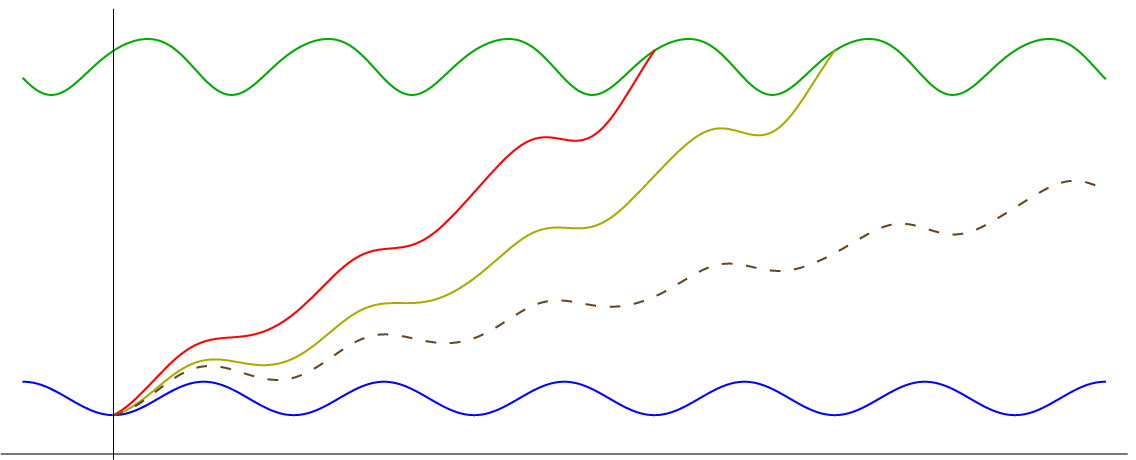}
\put(-18,89){\small $x_{\min}$}\put(-11,19){\small $\alpha$}\put(-150,50){\small $y_n$}\put(-88,57){\small $y_{n+1}$}\put(-208,-7){\scriptsize $0$}\put(-13,-7){\footnotesize $t\rightarrow$}\put(-213,86){\footnotesize $u$}\put(-213,96){\footnotesize $\uparrow$}\put(-11,48){\small $u$}\put(-100,-25){\small (a)}\hspace{1cm}  \includegraphics[scale=0.7]{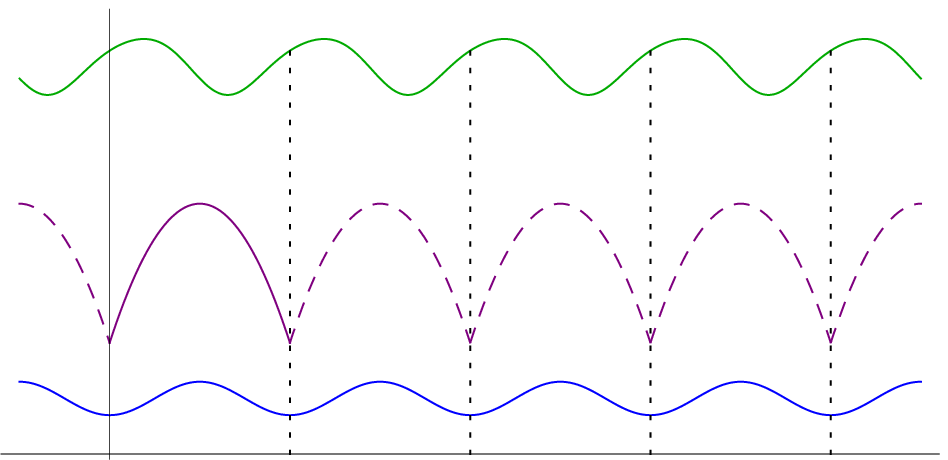}   \put(-135,-7){\scriptsize $T$}\put(-153,58){\small $y$} \put(-11,56){\small $\hat\alpha$}\put(-63,-7){\scriptsize $3T$}\put(-99,-7){\scriptsize $2T$}\put(-13,-7){\footnotesize $t\rightarrow$}\put(-18,89){\small $x_{\min}$}\put(-11,19){\small $\alpha$}\put(-171,-7){\scriptsize $0$}\put(-100,-25){\small (b)}\put(-180,21){\footnotesize $u_0$}\put(-176,86){\footnotesize $u$}\put(-176,96){\footnotesize $\uparrow$}
\end{center}\caption{(a) The decreasing sequence $y_n$ has some limit $u$. (b) The $T$-periodic extension of $y$ is a new lower solution $\hat\alpha$.}
\end{figure*}

\vspace{0.3cm}

{\em (Phase 1d)}\hspace{2cm}$u$ is asymptotic to $x_*$ in the
future.

\vspace{0.3cm}

To see this we consider the sequence of functions
$$u_n(t)=u(t+nT)\,,\hspace{2cm}t\in[0,T]\,,$$ which is uniformly
bounded and pointwise increasing, and made of solutions of
(\ref{eu0}). We use again Proposition \ref{prp1} to deduce that
$\{u_n\}$ converges in the $\mathcal C^2[0,T]$ topology to some
solution $u_*:[0,T]\to\mathbb R$ of (\ref{eu0}). This solution must
clearly be $T-$periodic and lie between $\alpha$ and $x_{\min}$. But
$x_{\min}$ was minimal, meaning that $u_*=x_{\min}$. The result
follows.

\vspace{0.2cm}

     {\bf (Step 2)} {From a particular case to Theorem \ref{th1} in its full generality}

\vspace{0.2cm}

    The first step was devoted to show Theorem \ref{th1} for $u_0=\alpha(0)$ and assuming that $f$ is bounded between $\alpha$ and $\beta$. We successively remove these restrictions in the phases below:

\vspace{0.3cm}

{\em (Phase 2a)}\ \ \ Let $f$ be bounded between $\alpha$ and $\beta$. Then, for any $u_0\in(\alpha(0),x_{\min}(0))$ there exists some lower
     solution $\hat\alpha$ of $(P)$ with $\alpha\leq\hat\alpha\leq
     x_{\min}$ and $\hat\alpha(0)=u_0$.

\vspace{0.3cm}

    Fix some $u_0\in(\alpha(0),x_{\min}(0))$ and consider the
     Dirichlet problem

     \begin{equation*}
 (D_{0})\equiv\begin{cases}
-\ddot y=f(t,y,\dot y)\,,\vspace{0.3cm}\\  y(0)=u_0=y(T)\,.
\end{cases}
\end{equation*}

For this problem, $\alpha$ is a lower solution, while $x_{\min}$ is
an upper solution. Thus, Proposition \ref{DP} states the existence
of some solution $y:[0,T]\to\mathbb R$ of $(D_{0})$ which lies
between $\alpha$ and $x_{\min}$.

We claim that $\dot y(0)>\dot y(T)$. Indeed, $\dot y(0)\not=\dot y(T)$ since otherwise
the $T$-periodic extension $\hat\alpha$ of $y$ would be a
solution of $(P)$, contradicting the minimality of $x_{\min}$. And
if $\dot y(0)<\dot y(T)$ then $\hat\alpha$ would be an upper solution
of $(P)$, implying the existence of some periodic solution between
$\alpha$ and $\hat\alpha$ and contradicting again the minimality of
$x_{\min}$. Thus, $\dot y(0)>\dot y(T)$ as claimed, and this means
that $\hat\alpha$ is a lower solution of $(P)$, see Fig. 3(b).

Step 1 now implies that for every $u_0\in[\alpha(0),x_{\min}(0))$ there exists some solution $u_r:[0,+\infty)\to\mathbb R$ with $u_r(0)=u_0$ and $\alpha\leq u_r\leq x_{\min}$ on $[0,+\infty)$ which is asymptotic to $x_{\min}$ in the future. But all this work was done assuming $f$ is bounded between $\alpha$ and $\beta$. The last part of the proof is devoted to see that {\bf [N]} is sufficient.

\vspace{0.3cm}

{\em (Phase 2b)}\ \ \ From boundedness to the Nagumo condition.

\vspace{0.3cm}

 Let now $f$ satisfy {\bf [N]}. We use the Modification Lemma \ref{ML} and transform equation (\ref{eu0}) accordingly. In view of the comments following the statement of the Lemma, the original and the modified equations have the same $T$-periodic solutions between $\alpha$ and $\beta$, and $x_{\min}$ continues to be the minimal one of them. The discussions before show that from each initial position $u_0\in[\alpha(0),x_{\min}(0))$ there starts some solution $u_r$ of our modified equation which lies between $\alpha$ and $\beta$ and is asymptotic to $x_{\min}$ in the future. But {\em (iii)} implies that they are indeed solutions of the original equation. The  proof is complete.

  \end{proof}
 The result which closes this Section is a straightforward consequence of Theorem \ref{th1} and Proposition \ref{prop1}, and states the instability of the maximal and the minimal solutions produced by the periodic method of lower and upper solutions. Since the Nagumo condition  {\bf [N]} always holds in the conservative case, it may be seen as a generalization of Proposition 3.1 of \cite{dan-ort}:
 \begin{corollary}\label{cor1}{Assume {\bf [N]} and uniqueness for initial value problems associated to (\ref{eu0}). Assume also that $\alpha$ (resp. $\beta$) is {\bf\em not} a solution of this equation. Then, $x_{\min}$ (resp. $x_{\max}$) is unstable, simultaneously in the past and in the future.}
 \begin{proof} Assume, for instance, that $\alpha$ is not a solution of (\ref{eu0}). Then, Proposition \ref{prop1}
 implies that $x_{\min}$, being a (non-strict) upper solution of $(P)$, verifies $\alpha(0)<x_{\min}(0)$. For every initial position $u_0\in\big[\alpha(0),x_{min}(0)\big)$, Theorem \ref{th1} states the existence of solutions
  $$u_l:(-\infty,0]\to\mathbb R,\hspace{2cm}  u_r:[0,+\infty)\to\mathbb R\,,$$
   which are asymptotic to $x_{\min}$ in the past and in the future respectively and verify (\ref{eu8}). By uniqueness of initial value problems,  $u_l<x_{min}$ and $u_r<x_{\min}$ on their respective domains. The result follows.

  \end{proof}
  \end{corollary}

  \section{The dynamics between two ordered solutions}\label{sec4}

In this Section $\alpha$ and $\beta$ are assumed to be {\em solutions} of the periodic problem $(P)$. In general this does not imply the existence of an unstable periodic solution between them, at least if by `instability' we mean Lyapunov instability {\em in the past  and in the future}. To check this assertion it suffices to consider the linear equation $-\ddot u=\dot u$. Observe that the periodic solutions are the constants, and all of them are unstable in the past, but stable in the future. We also notice that, at least for this example, the periodic solutions are always ordered and make up a continuous family.

Thus, we shall begin by setting ourselves under the opposite situations of $\alpha$ and $\beta$ being {\em neighboring} solutions of $(P)$; i.e., the only solutions $x$ of the periodic problem $(P)$ lying between $\alpha$ and $\beta$ are $x\equiv\alpha$ and $x\equiv\beta$. Neighboring solutions have previously been considered in
the literature, including for instance in the Aubry-Mather theory, which uses variational arguments to show the existence of heteroclinics between neighboring global periodic minimals (see, e.g. Theorem 2.6.2 of \cite{mos}). The periodic problem $(P)$ which occupies us now may not have a variational structure, and assuming that it has one, the neighboring solutions $\alpha\leq \beta$ may not be minimal. Indeed, one easily finds examples of neighboring periodic solutions which are not connected by heteroclinics (think for instance on two consecutive equilibria of some autonomous equation $-\ddot x=V'(x)$ lying at different potential levels). However, in this Section we shall see that one of the neighboring periodic solutions must be the limit of many asymptotic solutions, simultaneously in the past and in the future. Precisely:
\begin{theorem}\label{th2}{Let $\alpha\leq\beta$ be neighboring solutions of $(P)$ and assume {\bf [N]}.  Then, either $x\equiv\alpha$ or $x\equiv\beta$ has the following property: for any initial position $$u_0\in\Big(\alpha(0),\beta(0)\Big)\,,$$
there are solutions $u_l:(-\infty,0]\to\mathbb R,\ \
u_r:[0,+\infty)\to\mathbb R$ of (\ref{eu0}) with\
$$u_l(0)=u_0=u_r(0)\,,\hspace{1.5cm}\alpha\leq u_l\leq\beta\ \text{ on }(-\infty,0]\,,\hspace{1.5cm}\alpha\leq u_r\leq\beta\ \text{ on }[0,+\infty)\,,$$
which are asymptotic to $x$ in the past and in the future
respectively. }

 \begin{proof}

 We consider the set $\Sigma$ of solutions of (\ref{eu0}) on the time interval $[0,T]$ which start and end at the same position (with possibly different derivative) and lie between $\alpha$ and $\beta$:
  $$\Sigma:=\left\{u\in \mathcal C^2[0,T]\left|\begin{matrix}
\  u\text{ solves }(\ref{eu0})\\ \vspace{-0.4cm}\\
 \ u(0)=u(T)\\ \vspace{-0.4cm} \\
 \ \alpha\leq u\leq\beta\text{ on }[0,T]
 \end{matrix}\right.\right\}\,.$$

Observe that this set contains the solutions $\alpha$ and $\beta$. Since they are assumed to be neighboring, other solutions $u\in\Sigma$ must verify $\dot u(0)\not=\dot u(T)$. We start our proof by showing that $\Sigma$, which is  naturally endowed with $\mathcal C^2[0,T]$ topology, contains nontrivial connected subsets:

\vspace{0.2cm}

    \begin{claim}{\em For each $0<\epsilon<\frac{1}{2}(\beta(0)-\alpha(0))$ there exists a continuum $\Sigma_\epsilon\subset\Sigma$ such that
     $$\Big\{u(0):u\in\Sigma_\epsilon\Big\}=\Big[\alpha(0)+\epsilon,\beta(0)-\epsilon\Big]\,.$$

     }   \end{claim}
 To show the Claim we first replace our equation $(\ref{eu0})$ with a                                                                            modified one $(\tilde 1)$ as in Lemma \ref{ML}. From {\em (ii)} and {\em (iii)} we see that
  $$\Sigma=\left\{u\in \mathcal C^2[0,T]\left|\begin{matrix}
  u\text{ solves }(\tilde 1)\\ \vspace{-0.4cm}\\ u(0)=u(T)\in[\alpha(0),\beta(0)]
  \end{matrix}\right.\right\}\,.$$

   By the comments following the proof of Lemma \ref{ML} in the Appendix we see that $(\tilde 1)$ may actually be rewritten in the form (\ref{eu15}) for some globally bounded function $b$. Thus, for each $u_0\in[\alpha(0),\beta(0)]$, the solutions $u:[0,T]\to\mathbb R$ of (\ref{eu0}) which lie between $\alpha$ and
 $\beta$ and verify $u(0)=u(T)=u_0$ may be written as the fixed points of a compact, bounded operator $\mathcal F_{u_0}$ on the space of $\mathcal C^2[0,T]$ functions vanishing at times $t=0,T$. The family $\big\{\mathcal F_{u_0}\big\}_{u_0}$ is uniformly bounded and depends continuously on $u_0$, and the result then follows from the usual continuation properties of the Leray-Schauder topological degree.

           \vspace{0.1cm}

Having shown the Claim, let us complete the proof of the Theorem. We choose some sequence $\epsilon_n\searrow 0$ with $0<\epsilon_n<\big(\beta(0)-\alpha(0)\big)/2$ for every $n$. As observed before,  the associated connected sets $\Sigma_{\epsilon_n}$ do not contain periodic solutions, and we deduce that
for each $n$, either
\begin{enumerate}
\item[$(p_n^-)$] All solutions $u\in\Sigma_{\epsilon_n}$ verify that $\dot u(0)>\dot u(T)$, or
\item[$(p_n^+)$] All solutions $u\in\Sigma_{\epsilon_n}$ verify that $\dot u(0)<\dot u(T)$.
\end{enumerate}
 Observe that if $(p_n^-)$ holds for some $n$ then the elements of $\Sigma_{\epsilon_n}$ when extended to the real line by periodicity become lower solutions of $(P)$. On the contrary, if $(p_n^+)$ holds for some $n$ then the $T-$periodic extensions of the elements of $\Sigma_{\epsilon_n}$ make up upper solutions of $(P)$. We deduce that if there are infinitely many numbers $n$ for which $(p_n^-)$ holds, then problem $(P)$ has some lower solution starting from each initial position $\alpha_0\in(\alpha(0),\beta(0))$; moreover,this lower solution lies between $\alpha$ and $\beta$. On the contrary, if there are infinitely many numbers $n$ for which $(p_n^+)$ holds, then problem $(P)$ has some upper solution starting from each initial position $\beta_0\in(\alpha(0),\beta(0))$, and this upper solution lies between $\alpha$ and $\beta$. The result follows now from Theorem \ref{th1}.

  \end{proof}
\end{theorem}
An immediate consequence of Theorem \ref{th2}  is given below:
\begin{corollary}\label{cor2}{Assume {\bf [N]} and uniqueness of initial value problems. If $\alpha<\beta$ are neighboring solutions of the periodic problem $(P)$, one of them must be unstable, simultaneously in the past and in the future.}
\end{corollary}
We observe that, even though the method of lower and upper solutions also works for first order equations,
this result does not have an immediate extension there. For instance, $\alpha\equiv 0$ and $\beta\equiv 1$ are neighboring periodic solutions for $-\dot u=u(1-u)$, but $\alpha$ is stable in the future and $\beta$ is stable in the past. We owe this remark to R. Ortega.

 At this moment we come back to the question formulated at the beginning of this Section: given ordered solutions $\alpha<\beta$ of the periodic (second-order) problem $(P)$, is it possible to find an unstable solution between them? We already know that the answer is `yes' if $\alpha$ and $\beta$ are neighboring and `no' in general. In the spirit of the counterexample mentioned before we shall say that the periodic problem $(P)$ is  {\em degenerate} between $\alpha$ and $\beta$ if they may be linked by an increasing path of periodic solutions. Precisely, if there exists some $\mathcal C^{2,0}$ function $$\Psi:\mathbb R\times[0,1]\to\mathbb R\,,\hspace{1.5cm}(t,s)\mapsto\Psi(t,s)\,,$$ such that
\begin{enumerate}
\item[{\em (a)}]$\Psi(\cdot,0)=\alpha(\cdot),\ \ \Psi(\cdot,1)=\beta(\cdot)\,.$
\item[{\em (b)}]$\Psi(\cdot,s)$ is a solution of the periodic problem $(P)$ for any $s\in[0,1]$.
\item[{\em (c)}]$\Psi(t,s_1)<\Psi(t,s_2)$ if $s_1<s_2$.
\end{enumerate}

    This concept of degeneracy will be used to give a more precise answer to the question which opened this paragraph. We remark that the study of degeneracy is not new, for instance a related notion of degeneracy was considered in  \cite{ort-tar}.
\begin{proposition}\label{prrop}{Assume uniqueness for initial value problems and {\bf [N]}}. If $\alpha<\beta$ are ordered solutions of $(P)$, one of the following hold:
\begin{enumerate}
\item[($I$)] There exists some solution $x$ of the periodic problem $(P)$ which is unstable in the past and in the future and lies between $\alpha$ and $\beta$.

\item[($I\hspace{-0.1cm}I$)] $(P)$ is degenerate between $\alpha$ and $\beta$.
\end{enumerate}
  \begin{proof}
 We distinguish two possibilities; either $\alpha$ and $\beta$ enclose some ordered pair $\alpha\leq x_1<x_2\leq\beta$
  of neighboring periodic solutions, or not. In the first case Corollary \ref{cor2} immediately implies $(I)$; in the second case we are going to show $(I\hspace{-0.1cm}I)$.

  Thus, assume that between $\alpha$ and $\beta$ there are not pairs $x_1<x_2$ of neighboring periodic solutions. We consider the set $\mathcal F$ whose elements are totally ordered families $F\subset \mathcal C^2(\mathbb R)$  of solutions of the periodic problem $(P)$ lying between $\alpha$ and $\beta$. Here, `totally ordered' means that for any $F\in\mathcal F$ and $x_1\not=x_2\in F$ one has either $x_1<x_2$ or $x_2<x_1$.

  The set $\mathcal F$ itself may be endowed with the partial order provided by inclusion. Observe that if $\mathcal S\subset\mathcal F$ is any totally ordered subset, then $F:=\bigcup_{S\in\mathcal S}S$ is an upper bound for $\mathcal S$. Consequently, by Zorn's Lemma $\mathcal F$ must have some maximal element $F_0$.

\vspace{0.2cm}

\begin{claim}{The following hold:
     \begin{enumerate}
     \item[($\frak F_1$)] Every monotone sequence $\{x_n\}_n\subset F_0$ converges in the $\mathcal C^2$ norm to some $x_*\in F_0$.
     \item[($\frak F_2$)] For each initial position $x_0\in[\alpha(0),\beta(0)]$ there exists an unique element $x\in F_0$ with $x(0)=x_0$.
     \end{enumerate}}
     \end{claim}
    We first show the Claim. In view of the Modification Lemma \ref{ML} it suffices to prove {\em ($\frak F_1$)} assuming that $f$ is bounded between $\alpha$ and $\beta$. Now, the fact that every monotone sequence in $F_0$ converges in the $\mathcal C^2$ norm to some $T$-periodic solution follows immediately from Proposition \ref{prp1}, and the maximality of $F_0$ implies that this limit must belong to $F_0$.

    To see {\em ($\frak F_2$)} we first observe that, as a consequence of {\em ($\frak F_1$)}, the set
    $$A:=\Big\{x(0):x\in F_0\Big\}$$
    of initial positions of elements in $F_0$, must be closed. If $A$ were not the whole interval $[\alpha(0),\beta(0)]$ then it would mean the existence of pairs of neighboring solutions of $(P)$, which we are assuming that is not the case. Then,  for each initial position $x_0\in[\alpha(0),\beta(0)]$ there exists some element $x\in F_0$ with $x(0)=x_0$, and since two such functions should intersect tangentially, {\em ($\frak F_2$)} follows from the uniqueness of initial value problems associated to (\ref{eu0}).

  \vspace{0.1cm}

  To conclude the proof of Proposition \ref{prrop} we define
  $$\Psi(t,s):=x_s(t)\,,\hspace{2cm}t\in\mathbb R,\ s\in[0,1],$$
  where $x_s$ is the only element of $F_0$ with $x_s(0)=(1-s)\alpha(0)+s\beta(0)$. Then $\Psi$ satisfies the three conditions {\em (a),(b),(c)} above and the result follows.
  \end{proof}
\end{proposition}

When the equation is conservative, $-\ddot u=f(t,u)$, past and future Lyapunov stability become equivalent concepts. We close this Section with a result which is specific for the conservative case:
\begin{corollary}\label{cor3}{Assume that there is uniqueness for initial value problems associated to  $-\ddot u=f(t,u)$, and that $\alpha<\beta$ are ordered $T$-periodic solutions. Then there is some unstable solution between them (it may be either $\alpha$, or $\beta$, or other).}
\begin{proof} In view of Proposition \ref{prrop}, it suffices to consider the case in which the $T$-periodic problem associated to our equation is degenerate between $\alpha$ and $\beta$. Assume also that, for instance, $\alpha$ is stable and consider the solution $u$ verifying $u(0)=\alpha(0)$ and $\dot u(0)=\dot\alpha(0)+\epsilon$, where $\epsilon>0$ is small. On the maximal interval of time $\mathcal I$ where $u$ lies between $\alpha$ and $\beta$, it cannot cross tangentially a periodic solution, and then, the continuous function $s$ defined implicitly by $$u(t)=\Psi\big(t,s(t)\big)\,,\hspace{2cm}t\in\mathcal I\,,$$ must satisfy $\dot u(t)>\Psi_t(t,s(t))$ for any time $t$ (here, $\Psi=\Psi(t,s)$ is given by the definition of degeneracy). One easily deduces from here that $s$ is strictly increasing on $\mathcal I$. On the other hand, the stability of $\alpha$ implies that $\mathcal I=[0,+\infty)$ and $s(t)$ converges to some small limit $0<\ell<1$ as $t\to+\infty$. It follows that the periodic solution $\Psi(\cdot,\ell)$ is unstable and concludes the proof.
 \end{proof}
 \end{corollary}
 \section{Appendix: Discussing the statements of Section \ref{sec2}}\label{sec5}
       The last Section of this paper is devoted to sketch the proofs of some results which were presented in Section \ref{sec2} and used subsequently.
       \begin{proof}[Proof of Lemma \ref{ML}] It combines elements from the proofs of Theorem 5.4 of Chapter I of \cite{dec-hab} on the one hand and Lemma 4.2 of \cite{ort-rob} on the other.
        We start by letting
         \begin{equation}\label{eu11}
         \tilde f(t,u,\dot u):=-u+\gamma(t,u)+f\big(t,\gamma(t,u),\delta(\dot u)\big)\,,\hspace{1.5cm}(t,u,\dot u)\in\mathbb R^3\,,\end{equation}
       where $$\gamma(t,u):=\begin{cases}
\alpha(t) &\text{if }u\leq\alpha(t)\\
u &\text{if }\alpha(t)\leq u\leq\beta(t)\\
\beta(t) &\text{if }u\geq\beta(t)
\end{cases}\,,\hspace{1.5cm}
\delta(v):=\begin{cases}-K&\text{if }v\leq-K\\
v&\text{if }-K\leq v\leq K\\
K&\text{if }v\geq K\\
\end{cases}\,,
$$ and the constant $K>0$ is chosen in such a way that
\begin{equation}\label{eu12}
\int_0^K\frac{v}{\varphi(v)}\,dv>\beta(0)-\alpha(0)\,,\hspace{1.5cm}
K>\max\Big\{\|\dot\alpha\|_\infty,\|\dot\beta\|_\infty\Big\}\,.
\end{equation}

Clearly, $\tilde f$ is continuous, $T$-periodic and bounded between $\alpha$ and $\beta$. Also {\em (i)} is immediate from the definition of $f$, while {\em (iii)} is a consequence of the Nagumo condition (choose $\epsilon>0$ small enough so that $\int_\epsilon^Kv\,dv/\varphi(v)>\beta(0)-\alpha(0)$).

The proof of {\em (ii)} is less direct and will rely on three previous results which we describe next. The first and the third ones focus attention on the side limits
$$D_l h(t_0):=\underset{\begin{matrix}t\to t_0\\ t<t_0\end{matrix}}{\esslim }\dot h(t)\,,\hspace{2cm}D_r h(t_0):=\underset{\begin{matrix}t\to t_0\\ t>t_0\end{matrix}}{\esslim }\dot h(t)\,,$$
where $h$ is some Lipschitz-continuous function. The derivative $\dot h$ may exist only almost everywhere; thus, the limits above are `essential limits', meaning that they must be considered after $\dot h$ is perhaps redefined on a zero-measure set. For arbitrary Lipschitz-continuous functions $h$ these limits do not always exist,
but it cannot be the case when $h=\alpha$ is a lower solution of $(P)$:

\vspace{0.3cm}

{\bf {\em 1.\  Regularity of lower solutions\footnote{Of course one may formulate a corresponding version for upper solutions. We state it for the periodic problem because this is the framework where it will be used, but a quick glance at its proof shows that the result does not depend on the boundary conditions.}:\ \ }}{\em  Let $\alpha$ be a lower solution of (P). Then, $D_l\alpha(t_0)$ and $D_r\alpha(t_0)$ exist for every $t_0\in\mathbb R$ and satisfy  $D_l\alpha(t_0)\leq D_r\alpha(t_0)$.

\vspace{0.1cm}

\begin{proof}  Let $\omega\in W^{2,\infty}_{loc}(\mathbb R)$ be any second primitive of the function $\ddot\omega(t)= -f(t,\alpha(t),\dot\alpha(t))$, and let $h:=\omega-\alpha$. In this way, $h$ is locally Lipschitz-continuous and $-\ddot h\geq 0$ in the distributional sense. But positive distributions are positive, locally finite Borel measures (see e.g. Theorem  6.22 of \cite{lie-los}). Consequently, $\dot h$ is decreasing and has side limits at each point. The result follows.
  \end{proof}                          }

 A second consequence of the fact that positive distributions are positive, locally finite Borel measures is a distributional version of the maximum principle in one dimension:

\vspace{0.3cm}

{\bf {\em 2.\ Maximum principle:}} {\em Let $h:(a,b)\to\mathbb R$ be Lipschitz-continuous and verify
$$-\ddot h\geq 0\text{ on }(a,b)\,,$$
in the distributional sense. If $h$ attains a local minimum at some point $t_0\in(a,b)$ then it is constant.}

 \vspace{0.1cm}
 \begin{proof}
 As before, $\dot h$ is a decreasing function. It implies the result.\end{proof}

The third result is an elementary observation which follows directly from our definition of the side limits $D_lh(t_0)$ and $D_rh(t_0)$. It gives information on these limits at any point $t_0$ where $h$ has a local minimum.

\vspace{0.3cm}

{\bf {\em 3.\ $D_lh$ and $D_rh$  at local minima:}} {\em Let $h:(a,b)\to\mathbb R$ be Lipschitz-continuous and attain a local minimum at $t_0\in (a,b)$. Assume also that the limits $D_lh(t_0)$ and $D_rh(t_0)$ exist. Then
$D_lh(t_0)\leq 0\leq D_rh(t_0).$
}

\vspace{0.1cm}
 \begin{proof}
 Integrating from $t_0$ to $t$ we see that, if $D_lh(t_0)$ exists, then
   $$\lim_{\begin{matrix}t\to t_0\\ t<t_0\end{matrix}}\frac{h(t)-h(t_0)}{t-t_0}=D_lh(t_0)\,,$$
i.e., the left-side derivative of $h$ at $t_0$ exists and has the same value. Similarly, if $D_rh(t_0)$ exists at some point $t_0$, then the right-side derivative there also exists and coincides with $D_rh(t_0)$. The result follows.
 \end{proof}

We complete now the proof of statement {\em (ii)} of Lemma \ref{ML}. We use a contradiction argument and assume, for instance, that $u:I\to\mathbb R$ is a solution of $(\tilde 1)$ such that $h:=u-\alpha$ attains a local minimum at some interior point $t_0\in I$ with $u(t_0)<\alpha(t_0)$. Since $u$ is a $\mathcal C^1$ function, result {\bf{\em 1.}} above implies that the side limits $D_l h(t_0),\ D_l h(t_0)$ exist and moreover $D_lh(t_0)\geq D_rh(t_0)$. But $t_0$ being a point where $h$ attains a local minimum, {\bf\em 3.} gives that $D_lh(t_0)= D_rh(t_0)=0$, or what is the same,
\begin{equation}\label{eu13}
\lim_{t\to t_0}\dot\alpha(t)=\dot u(t_0)\,.
\end{equation}

In particular, $|\dot u(t_0)|\leq\|\dot\alpha\|_\infty$, and for any $t$ on some neighborhood of $t_0$ one has
\begin{equation*}
u(t)<\alpha(t)\,,\hspace{2cm}|\dot u(t)|\leq K\,,
\end{equation*}
and then $-\ddot u(t)=-u(t)+\alpha(t)+f(t,\alpha(t),\dot u(t))$. Consequently, on this same neighborhood we have
$$-\ddot h(t)=-\ddot u(t)+\ddot\alpha(t)\geq-u(t)+\alpha(t)+f(t,\alpha(t),\dot u(t))-f(t,\alpha(t),\dot\alpha(t))\,,$$
and combining (\ref{eu13}) with the continuity of $f$ we see that
\begin{equation}\label{eu14}
-\ddot h(t)\geq\frac{1}{2}\,\big(-u(t_0)+\alpha(t_0)\big)>0
 \end{equation}
 on some (possibly smaller) neighborhood of $t_0$. But $h$ attains a local minimum at $t_0$, so that in view of  {\bf{\em 2.}} above it should be constant on this neighborhood. Then, $-\ddot h\equiv0$ there, contradicting (\ref{eu14}). This contradiction proves the result.

 \end{proof}

The proof of Lemma \ref{ML} actually gives some additional information which was not stated in Section \ref{sec2}. It implies that the nonlinearity $\tilde f$ may be chosen in such a way that $b(t,u,\dot u):=\tilde f(t,u,\dot u)+u$ is {\em globally bounded} on $\mathbb R^3$. This allows us to rewrite $(\tilde 1)$ as
\begin{equation}\label{eu15}
-\ddot u=-u+ b(t,u,\dot u)\,.
\end{equation}
We are led to the existence result for the periodic method of lower and upper solutions.
\begin{proof}[Proof of Proposition \ref{prop0}]
Schauder's Fixed-Point Theorem guarantees the existence of some $T$-periodic solution to (\ref{eu15}). Since this equation is equivalent to $(\tilde 1)$, in view of the comments following the statement of Lemma \ref{ML} in Section \ref{sec2} we see that this is a $T$-periodic solution of (\ref{eu0}) lying between $\alpha$ and $\beta$. It implies the result.
\end{proof}

\begin{proof}[Proof of Proposition \ref{prp1}]   This result is very standard. Our assumptions imply that $\{u_n\}$ and $\{\ddot u_n\}$ are uniformly bounded sequences. Then, one easily checks that $\{\dot u_n\}$ must also be uniformly bounded and we conclude that $\{u_n\}$ and $\{\dot u_n\}$ are equicontinuous. Ascoli-Arzela's lemma now implies that $\{u_n\}\to u$ in the $\mathcal C^1[0,T]$ topology. But since $u_n$ is a solution of (\ref{eu0}) for each $n$ we deduce that $u$ solves the same equation and $\{u_n\}\to u$ in the $\mathcal C^2[0,T]$ topology. The result follows.
\end{proof}
\begin{proof}[Proof of Proposition \ref{upp-low}] Clearly we may focus attention on the part of the statement concerning $x_{\max}$. On the other hand, after possibly modifying $f$ as in Lemma \ref{ML} we may assume that it is bounded between $\alpha$ and $\beta$. We consider the set $\Gamma$ of periodic solutions which lie between them:
$$\Gamma:=\left\{x\in\mathcal C^2(\mathbb R)\left|\begin{matrix}
x\text{ solves }(P)\\ \vspace{-0.4cm}\\
\alpha\leq x\leq\beta
\end{matrix}\right.\right\}$$

 This set is naturally ordered by the pointwise order $\leq$. Using Proposition \ref{prp1} we see that every totally ordered subset
 $A\subset\Gamma$ has an upper bound in $\Gamma$. Then, Zorn's lemma implies the existence of some maximal element $x_{\max}$. It verifies
 \begin{equation}\label{eu16}
 x\in\Gamma,\ \ x\geq x_{\max}\ \Rightarrow\  x\equiv x_{\max}\,.\end{equation}

 Using a contradiction argument assume now that there exists $x_*\in\Gamma$ with $$x_*(t_0)>x_{\max}(t_0)$$ for some $t_0\in\mathbb R$. Since $x_*\ngeq x_{\max}$, we may find times $t_-<t_0<t_+<t_-+T$ such that
 $$x_*(t_\pm)=x_{\max}(t_\pm)\,,\hspace{2cm}x_*(t)>x_{\max}(t)\ \forall t\in[t_-,t_+]\,.$$
We define $\alpha_*:\mathbb R\to\mathbb R$ by
$$\alpha_*(t):=\begin{cases}
x_*(t)\text{ if }t\in[t_-,t_+]\,,\\
x_{\max}(t)\text{ if }t\in(t_+,t_-+T)\,,\end{cases}$$
and extended by periodicity. A simple integration-by-parts argument shows that $\alpha_*$ is a new lower solution of $(P)$. Since $\alpha_*\leq\beta$, Proposition \ref{prop0} implies the existence of some $\tilde x\in\Gamma$ with $\tilde x\geq\alpha_*$, contradicting (\ref{eu16}). This concludes the proof.

\end{proof}

The proofs of Propositions \ref{DP} and \ref{prop}  use similar arguments to the proof of Proposition \ref{upp-low}, and we shall not reproduce them here. We conclude the paper with the proof of Proposition \ref{prop1}, which follows along the lines of Lemma 2.4 of \cite{jac-sch}.
\begin{proof}[Proof of Proposition \ref{prop1}] Using a contradiction argument, we assume that $\alpha\leq\beta$ is an ordered pair of lower and upper solutions of $(P)$ verifying
$$\alpha(t_0)<\beta(t_0)\,,\hspace{2cm}\alpha(t_1)=\beta(t_1)$$
for some  $t_0,t_1\in\mathbb R$. Using Lemma \ref{ML} it is not restrictive to assume that $f$ is bounded between $\alpha$ and $\beta$; moreover, the periodicity of $\alpha,\beta$ allows us to assume $t_0<t_1<t_0+T$. For any $y_0\in(\alpha(t_0),\beta(t_0))$ we consider the Dirichlet problem
$$(D_{y_0})\equiv\begin{cases}
-\ddot y=f(t,y,\dot y)\vspace{0.3cm}\\  y(t_0)=y(t_0+T)=y_0
\end{cases}\,.$$

Fix now points
$y^{-}_0<y_0^+\in(\alpha(t_0),\beta(t_0))$. Since
$\alpha$ and $\beta$ are lower and upper solutions for
$(D_{y_0^-})$, Proposition \ref{DP} implies that this problem
must have some solution $y^-$ lying between $\alpha$ and $\beta$.
In turn,  $y^-$ and $\beta$ are respectively lower and upper
solutions for  $(D_{y_0^+})$, which should have a solution
$y^+$ lying between $y^-$ and $\beta$. Now, $y^-$ and
$y^+$ coincide at $t=t_1$,  and being ordered, they must be
tangent there. This contradicts the uniqueness of initial value
problems and concludes the proof.
\end{proof}
               

\begin{thebibliography}{xx}
     \bibitem{dan-ort}Dancer, E.N.; Ortega, R.,
    {\em The index of Lyapunov stable fixed points in two dimensions.}
J. Dynam. Differential Equations 6 (1994), no. 4, 631--637.
\bibitem{dec-hab} De Coster, C.; Habets, P., {\em Two-point boundary value problems: lower and upper solutions}. Mathematics in Science and Engineering, 205. Elsevier B. V., Amsterdam, 2006.
\bibitem{jac-sch} Jackson, L.K.; Schrader, K.W., {\em
Comparison theorems for nonlinear differential equations.}
J. Differential Equations 3 (1967) 248--255.
\bibitem{lie-los}Lieb, E.H.; Loss, M., {\em Analysis.}
Second edition. Graduate Studies in Mathematics, 14. American Mathematical Society, Providence, RI, 2001. \bibitem{mos} Moser, J.
{\em Selected Chapters in the Calculus of Variations.}
Lecture notes by Oliver Knill. Lectures in Mathematics ETH Z\"{u}rich. Birkh\"{a}user Verlag, Basel, 2003. \bibitem{ort-rob}Ortega, R.; Robles-P\'erez, A.M.,
{\em A maximum principle for periodic solutions of the telegraph equation.}
J. Math. Anal. Appl. 221 (1998), no. 2, 625--651.
\bibitem{ort-tar} Ortega, R.; Tarallo, M.,
{\em Degenerate equations of pendulum-type.}
Commun. Contemp. Math. 2 (2000), no. 2, 127--149.
\bibitem{pal-mel}Palis, J.; de Melo, W., {\em Geometric theory of dynamical systems. An introduction.} Springer-Verlag, New York-Berlin, 1982.

\bibitem{ure}Ure\~{n}a, A.J., {\em All periodic minimizers are unstable.}
Arch. Math. (Basel) 91 (2008), no. 1, 63--75.
\bibitem{ure2}Ure\~{n}a, A.J., {\em Invariant manifolds around equilibria of Newtonian equations: Some pathological examples.}
J. Differential Equations, 249 (2010), no. 2, 366--391.
\end{thebibliography}
\end{document}